\documentclass[11pt,reqno]{amsart}
\usepackage{amsmath,amsfonts,amssymb,amsthm,amscd,comment,euscript}
\usepackage[all]{xy}
\usepackage{graphicx}
\usepackage{enumerate} 
\usepackage[colorlinks=true]{hyperref} 
\usepackage[usenames,dvipsnames]{xcolor}
\usepackage{enumitem}

\usepackage{bbm}

 \calclayout

\swapnumbers

\theoremstyle{plain}
\newtheorem{lem}{Lemma}[section]

\newtheorem{prop}[lem]{Proposition}
\newtheorem{thm}[lem]{Theorem}

\theoremstyle{definition}

\newtheorem{rem}[lem]{Remark}
\newtheorem{dfn}[lem]{Definition}

\newcommand{\RHomV}{\textbf{R}\mathrm{Hom}^{\cc V}}
\newcommand{\HomV}{\mathrm{Hom}^{\cc V}}
\newcommand{\cc}{\mathcal}
\newcommand{\cofreg}{\cof_{\mathrm{reg}}}
\newcommand{\cof}{-\mathrm{cof}}

\newcommand{\twothree}{\aleph}
\newcommand{\twothreemap}[1]{h}
\newcommand{\twothreecone}[1]{C}
\newcommand{\colim}[1]{\underset{#1}{\mathsf{colim}}}

\newcommand{\LSM}{L^{\cc V}_S\cc M}

\begin{document}
\footskip30pt

\baselineskip=1.1\baselineskip

\title{A Criterion for the Monoid Axiom in Enriched Bousfield Localizations}
\email{peterjbonart@gmail.com}

\author{Peter Bonart}

\begin{abstract}
	This paper proves a criterion for verifying the monoid axiom in enriched left Bousfield localizations.
\end{abstract}


\maketitle
\thispagestyle{empty}
\pagestyle{plain}

\tableofcontents

\section{Introduction}
This paper proves a criterion for verifying the monoid axiom in enriched left Bousfield localizations.

The monoid axiom is introduced in \cite{schwede2000algebras}. Enriched Bousfield localizations are introduced in \cite{barwick2010bousfield}.
The main theorem of this paper is as follows.

\begin{thm}\label{thmIntroduction}
	Let $\cc V$ be a monoidal model category with a set of generating cofibrations $I_{\cc V}$. Let $\cc M$ be a monoidal model $\cc V$-category with a set of weakly generating trivial cofibrations $J^\prime$ (e.g. a set of generating trivial cofibrations).
	Let $S$ be a class of cofibrations with cofibrant domain in $\cc M$, such that the $\cc V$-Bousfield localization $L^{\cc V}_S\cc M$ exists.
	Assume that $J^\prime \cup (S \square I_{\cc V}) $ permits the small object argument.
	If all morphisms in
	$$((J^\prime \cup (S \square I_{\cc V}) )\otimes \cc M)\cofreg $$ are weak equivalences in $\LSM$, then $\LSM$ is a monoidal model $\cc V$-category that satisfies the monoid axiom.
\end{thm}

The rest of the introduction will recall the notions that appear in the theorem and motivate the theorem.

\begin{dfn}
	Let $\cc C$ be a cocomplete category.
	\begin{enumerate}
		\item Let $I$ be a class of maps in $\cc C$. We define $I\cofreg$ to be the class of all transfinite compositions of base changes of morphisms from $I$.
		\item Let $I$ be a class of maps in $\cc C$. We say that $I$ \textit{permits the small object argument} if domains of morphisms from $I$ are small relative to $I\cofreg$.
		\item Let $I, J$ be two classes of maps in $\cc C$. Then $I \square J$ denotes the set of all pushout-products from maps from $I$ and maps from $J$.
	\end{enumerate}
\end{dfn}

\begin{dfn}
	A model category $\cc M$ with class of trivial cofibrations $TC$ and a monoidal structure $\otimes$ is said to \textit{satisfy the monoid axiom} if all morphisms from
	$$(TC \otimes \cc M)\cofreg $$
	are weak equivalences in $\cc M$.
\end{dfn}

If $\cc M$ has a set of generating trivial cofibrations $J$, the following well-known lemma can help verify the monoid axiom in $\cc M$.

\begin{lem}
	Let $\cc M$ be a model category and monoidal category. Suppose $M$ is cofibrantly generated with a set of generating cofibrations $J$.
	If all morphisms
	$$(J \otimes \cc M)\cofreg$$
	are weak equivalences in $M$, then $M$ satisfies the monoid axiom.
\end{lem}

This lemma is only really useful if one has an easy explicit description of $J$.
In Bousfield localizations this lemma is usually not very useful, because it is usually very difficult to explicitly describe a set of generating trivial cofibrations $J$ for a Bousfield localization. Usually the only description of $J$ is that $J$ consists of all trivial cofibrations between $\kappa$-small objects for some cardinal $\kappa$ whose existence is proven by the Bousfield-Smith cardinality argument. With this description, verifying the monoid axiom for $J$ is hardly any easier than just directly verifying the monoid axiom for all trivial cofibrations.

We introduce the following notion of \textit{weakly generating trivial cofibrations}, which turn out to be much more helpful for verifying the monoid axiom in Bousfield localizations.
\begin{dfn}
	Let $\cc M$ be a model category and let $J^\prime$ be a set of trivial cofibrations in $\cc M$. We say that $J^\prime$ is a set of \textit{weakly generating trivial cofibrations} for $\cc M$, if for every morphism $f : A \rightarrow B$ in $\cc M$, if $f$ has the right lifting property with respect to $J^\prime$ and $B$ is fibrant, then $f$ is a fibration.
\end{dfn}
The definition is inspired by the weakly finitely generated model categories from \cite[Definition 3.4]{dundas2003enriched}.

Weakly generating trivial cofibrations do not characterize all fibrations through lifting properties, but only those fibrations that have fibrant codomain. 

If we have a model category $\cc M$ with a set of generating trivial cofibrations $J$, then
for every Bousfield localization of $\cc M$ we can explicitly describe a set of weakly generating trivial cofibrations.

\begin{prop}
	Let $\cc V$ be a monoidal model category with a set of generating cofibrations $I_{\cc V}$.
	Let $\cc M$ be a monoidal model $\cc V$-category with a set of weakly generating trivial cofibrations $J^\prime$ (e.g. a set of generating trivial cofibrations).
	Let $S$ be a class of cofibrations with cofibrant domain in $\cc M$, such that the $\cc V$-Bousfield localization $L^{\cc V}_S\cc M$ exists.
	Then the set
	$$J^\prime \cup (I_{\cc V} \square S) $$
	is a set of weakly generating trivial cofibrations for $L^{\cc V}_S\cc M$.
\end{prop}
The proof of this proposition is in Section \ref{sectionWgtc}.

To verify the monoid axiom it is enough to verify it on a set of weakly generating trivial cofibrations.

\begin{thm}
	Let $\cc M$ be a model category that is also a monoidal category. Let $J^\prime$ be a set of weakly generating trivial cofibrations for $\cc M$ that permits the small object argument. Then the monoid axiom can be checked on $J^\prime$ in the sense, that if
	$$(J^\prime \otimes \cc M)\cofreg $$ consists of weak equivalences, then $\cc M$ satisfies the monoid axiom.
\end{thm}

The proof of this theorem is in Section \ref{sectionMonoidAxiom}.

By combining the above two results we obtain Theorem \ref{thmIntroduction}.
The proof of Theorem \ref{thmIntroduction} is in Section \ref{sectionMonoidInBousfield}.

In a previous paper \cite[Section 3]{bonart2023triangulated} we have applied arguments similar to the ones from this paper to prove the monoid axiom in a specific situation. This paper generalizes the arguments that were used to prove the monoid axiom in that paper, by removing all finiteness conditions and enriching in an arbitrary cofibrantly generated monoidal model category $\cc V$.

\section{Fibrations in Enriched Bousfield localizations}

From now on let $\cc V$ be a monoidal model category. Let $I_{\cc V}$ be a set of generating cofibrations for $\cc V$. Let $\cc M$ be a model $\cc V$-category. Let $S$ be a class of cofibrations between cofibrant objects in $\cc M$, such that the enriched Bousfield localization $\LSM$ exists.

We will need to adapt a few lemmas from the theory of normal Bousfield localizations from \cite{hirschhorn2003model} to enriched Bousfield localizations.

The following definition is a slightly corrected version of \cite[Definition 7.6.3]{hirschhorn2003model}.
The definition needs to be slightly corrected, because the category defined in \cite[Definition 7.6.3]{hirschhorn2003model} is often not connected, and then \cite[Theorem 7.6.4 (3)]{hirschhorn2003model} is false.
\begin{dfn}
	For a morphism $f : A\rightarrow B$ in $\cc M$, we define a category $(A \downarrow \cc M \downarrow B)_f$ as follows:
	Its objects are diagrams $A \overset{g}{\rightarrow} X \overset{h}{\rightarrow} B$ in $\cc M$ satisfying $hg = f$.
	A morphism from $A \rightarrow X \rightarrow B$ to $A \rightarrow Y \rightarrow B$ is a morphism $X \rightarrow Y$ in $\cc M$ such that the following diagram commutes:
	$$\xymatrix{ A \ar[r] \ar[d] & X \ar[dl] \ar[d] \\
	             Y \ar[r] & B }$$
             
    Note that $(A \downarrow \cc M \downarrow B)_f$ forms a model category in which a morphism $X \rightarrow Y$ is a weak equivalence, fibration or cofibration if it is one in $\cc M$.
\end{dfn}

The following Proposition is a straightforward adaption of \cite[Proposition 3.3.15 (1)]{hirschhorn2003model} to enriched Bousfield localizations.
\begin{lem}\label{lemhirsch3315}
	Suppose we have a commutative triangle in $\cc M$
	$$\xymatrix{ X \ar[dr]_p \ar[rr]^f& & Y \ar[dl]^q  \\
	             & Z}$$
    such that $p$ is a fibration in $\cc M$ and $f$ is a weak equivalence in $\cc M$ and $q$ is a fibration in $\LSM$.
    Then $p$ is a fibration in $\LSM$.
\end{lem}
\begin{proof}
	We show that $p$ has the right lifting property with respect to trivial cofibrations in $\LSM$.
	Let $g : A \rightarrow B$ be a trivial cofibration in $\LSM$, and consider a diagram
	$$\xymatrix{A \ar[d]_g \ar[r]^a & X \ar[d]^p \\
	            B \ar[r]_b  & Z  } $$
    Since $q$ is a fibration in $\LSM$ we can get a lift in the diagram
    $$\xymatrix{A \ar[d]_g \ar[r]^{fa} & Y \ar[d]^q \\
    	B \ar[r]_b \ar@{..>}[ur]  & Z  } $$
    
    So we have a map $B \rightarrow Y$ in the category $(A \downarrow \cc M \downarrow Z)_{bg}$.
    Since $B$ is cofibrant in $(A \downarrow \cc M \downarrow Z)_{bg}$, and the map $f : X \rightarrow Y$ is a weak equivalence between fibrant objects in $(A \downarrow \cc M \downarrow Z)_{bg}$, it follows by \cite[Corollary 7.7.5]{hirschhorn2003model} that there also exists a map $B \rightarrow X$ in $(A \downarrow \cc M \downarrow Z)_{bg}$. This map then solves the lifting problem
    $$\xymatrix{A \ar[d]_g \ar[r]^a & X \ar[d]^p \\
    	B \ar[r]_b \ar@{..>}[ur]  & Z  } $$
    and then $p$ is a fibration in $\LSM$.
\end{proof}

The following Proposition is a straightforward adaption of \cite[Proposition 3.3.16 (1)]{hirschhorn2003model} to enriched Bousfield localizations.
\begin{prop}\label{lemhirsch3316}
	Let $p : X \rightarrow Y$ be a morphisms in $\cc M$, such that $p$ is a fibration in $\cc M$ and $X$ and $Y$ are fibrant in $L^{\cc V}_S\cc M$.
	Then $p$ is a fibration in $L^{\cc V}_S\cc M$.
\end{prop}
\begin{proof}
	Factor $p$ as $X \overset{f}{\rightarrow} W \overset{q}{\rightarrow} Y$, where $f$ is a trivial cofibration in $L^{\cc V}_S\cc M$ and $q$ is a fibration in in $\LSM$.
	Since $q$ is a fibration and $Y$ is fibrant in $\LSM$, it follows that $W$ is fibrant in $\LSM$.
	Then $f$ is a weak equivalence between fibrant objects in $\LSM$.
	Since the identity functor $Id_{\cc M} : \cc M \rightarrow \cc M$ is a right Quillen functor from $\LSM$ to $\cc M$, it preserves weak equivalences between fibrant objects. Since $f$ is a weak equivalence between fibrant objects in $\LSM$ it follows that $f$ is also a weak equivalence in $\cc M$.
	The result now follows by applying Lemma \ref{lemhirsch3315} to the following triangle.
	$$\xymatrix{ X \ar[dr]_p \ar[rr]^f& & W \ar[dl]^q  \\
		& Y}$$
\end{proof}

\begin{lem}\label{lemlocalfibrantifrlp}
	Let $\cc M$ be a model $\cc V$-category.
	Let $S$ be a class of cofibrations with cofibrant domain in $\cc M$, such that the $\cc V$-Bousfield localization $L^{\cc V}_S\cc M$ exists.
	Then an object $X \in \cc M$ is fibrant in  $L^{\cc V}_S\cc M$ if and only if $X$ is fibrant in $\cc M$ and $X \rightarrow 1$ has the right lifting property with respect to $I_{\cc V} \square S$.
\end{lem}
\begin{proof}
	Let $X$ be an object that is fibrant in $\cc M$.
	Then $X$ is fibrant in $L^{\cc V}_S\cc M$ if and only if $X$ is $\cc V$-enriched $S$-local. This means that for every $s : A \rightarrow B$ with $s \in S$ the map
	$$s^* :  \RHomV(B,X) \rightarrow \RHomV(A,X) $$
	is a weak equivalence in $\cc V$.
	Since $X$ is fibrant in $\cc M$ and $s$ is a cofibration between cofibrant objects, this is the case if and only if the map
	$$s^* : \HomV(B,X) \rightarrow \HomV(A,X)$$
	between the non-derived hom objects is a weak equivalence in $\cc V$.
	Since $s$ is a cofibration and $X$ is fibrant in $\cc M$, the map $s^*$ is a fibration in $\cc V$.
	So $s^*$ is a weak equivalence if and only if $s^*$ is a trivial fibration. This is the case if and only if $s^*$ has the right lifting property with respect to $I_{\cc V}$.
	Now for any $f : V \rightarrow W$ with $f \in I_{\cc V}$, a diagram
	$$\xymatrix{  V \ar[r] \ar[d]_f & \HomV(B,X) \ar[d]^{s^*} \\
	             W  \ar[r] \ar@{..>}[ur] & \HomV(A,X)  }$$
	has a lift if and only if the diagram
	$$\xymatrix{  W \otimes A \underset{V \otimes A}{\coprod} V \otimes B \ar[r] \ar[d]_{f \square s } & X \ar[d] \\
		W \otimes B \ar[r] \ar@{..>}[ur] & 1  }$$
	has a lift.
	Therefore $X$ is $S$-local if and only if $X \rightarrow 1$ has the right lifting property with respect to $I_{\cc V} \square S$.
\end{proof}

\section{Weakly Generating Trivial Cofibrations}\label{sectionWgtc}

The following definition is inspired by the weakly finitely generated model categories from
\cite[Definition 3.4]{dundas2003enriched}.
However unlike them we do not impose any finiteness conditions.
\begin{dfn}
	Let $\cc M$ be a model category and let $J^\prime$ be a set of trivial cofibrations in $\cc M$. We say that $J^\prime$ is a set of \textit{weakly generating trivial cofibrations} for $\cc M$, if for every morphism $f : A \rightarrow B$ in $\cc M$, if $f$ has the right lifting property with respect to $J^\prime$ and $B$ is fibrant, then $f$ is a fibration.
\end{dfn}

\begin{lem}\label{LemmaJcof}
	Let $\cc M$ be a model category, and $J^\prime$ a set of weakly generating trivial cofibrations that permits the small object argument.
	Let $f : A \rightarrow B$ be a trivial cofibration, and assume that $B$ is fibrant.
	Then $f \in J^\prime\cof$.
\end{lem}
\begin{proof}
	According to the small object argument \cite[Theorem 2.1.14]{hovey2007model} we can factor $f$ as $f = qi$ with $i \in J^\prime\mathrm{-cof}_{\mathrm{reg}}$ and $q$ having the right lifting property with respect to $J^\prime$.
	$$\xymatrix{A \ar[dr]_i \ar[rr]^f & & B \\
		& Z \ar[ur]_{q} &}  $$
	
	Since $q$ has a fibrant codomain and $q$ has the right lifting property with respect to $J^\prime$, it follows that $q$ is a fibration.
	Then $f$ has the left lifting property against $q$ so from the retract argument \cite[Lemma 1.1.9]{hovey2007model} it follows that $f$ is a retract of $i$. Since $i \in J^\prime\mathrm{-cof}_{\mathrm{reg}}$ this implies $f \in J^\prime\mathrm{-cof}$.
\end{proof}

\begin{rem}
	If $\cc M$ is a model category and $J$ is a set of generating trivial cofibrations for $\cc M$, then $J$ is also a set of weakly generating trivial cofibrations for $\cc M$. However the converse does not hold, and not every set of weakly generating trivial cofibrations is a set of generating trivial cofibrations.
\end{rem}

It is generally difficult to explicitly describe a set of generating trivial cofibrations in a Bousfield localization. However we can always explicitly describe a set of weakly generating trivial cofibrations in a Bousfield localization.
\begin{prop}\label{lemmaBousfield}
	Let $\cc M$ be a monoidal model $\cc V$-category with a set of weakly generating trivial cofibrations $J^\prime$.
	Let $S$ be a class of cofibrations with cofibrant domain in $\cc M$, such that the $\cc V$-Bousfield localization $L^{\cc V}_S\cc M$ exists.
	Then the set
	$$J^\prime \cup (I_{\cc V} \square S) $$
	is a set of weakly generating trivial cofibrations for $L^{\cc V}_S\cc M$.
\end{prop}
\begin{proof}
	All morphisms from $J^\prime$ and $S$ are trivial cofibrations in $L^{\cc V}_S\cc M$, and all morphisms from $I_{\cc V}$ are cofibrations in $\cc V$, so the set $J^\prime \cup (I_{\cc V} \square S)$ consists out of trivial cofibrations in $\cc M$.
	
	Let $f : A \rightarrow B$ be a morphism in $\cc M$ such that $f$ has the right lifting property with respect to $J^\prime \cup (I_{\cc V} \square S)$ and $B$ is fibrant in $L^{\cc V}_S\cc M$.
	Then $B$ is also fibrant in $\cc M$, and $f$ has the right lifting property with respect to $J^\prime$. Therefore $f$ is a fibration in $\cc M$, and then $A$ is fibrant in $\cc M$.
	
	By Lemma \ref{lemlocalfibrantifrlp} the map $B \rightarrow 1$ has the right lifting property with respect to $I_{\cc V} \square S$. Since $A \overset{f}{\rightarrow} B$ and $B \rightarrow 1$ have the right lifting property with respect to $I_{\cc V} \square S$, it follows that $A \rightarrow 1$ has the right lifting property with respect to $I_{\cc V} \square S$. Lemma  \ref{lemlocalfibrantifrlp} now implies that $A$ is fibrant in $L^{\cc V}_S\cc M$,
	and then Lemma \ref{lemhirsch3316} implies that $f$ is fibration in $L^{\cc V}_S\cc M$.
	Thus the set
	$J^\prime \cup (I_{\cc V} \square S) $
	is a set of weakly generating trivial cofibrations for $L^{\cc V}_S\cc M$.
\end{proof}

\section{Monoid Axiom}\label{sectionMonoidAxiom}

\begin{thm}\label{ThmCheckOnJPrime}
	Let $\cc M$ be a model category that is also a monoidal category ($\cc M$ need not be a monoidal model category). Let $J^\prime$ be a set of weakly generating trivial cofibrations for $\cc M$ that permits the small object argument. Then the monoid axiom can be checked on $J^\prime$ in the sense, that if
	$$(J^\prime \otimes \cc M)\cofreg $$ consists of weak equivalences, then $$(TC \otimes \cc M)\cofreg $$ consists of weak equivalences, where $TC$ is the set of all trivial cofibrations of $\cc M$.
\end{thm}

The proof of this theorem relies on a technical lemma.

\begin{dfn}
	Let $\cc C$ be a category, $I$ a class of morphisms in $\cc C$.
	We write $\twothree(I)$ for the class of all morphisms $f : A \rightarrow B$ in $\cc C$ such that there exists a morphism $\twothreemap{f} : B \rightarrow \twothreecone{f}$, such that $\twothreemap{f} \in I$ and $\twothreemap{f} \circ f \in I$.
	$$\twothree(I) := \{f \in \mathrm{Mor}(\cc C) \mid \exists \twothreemap{f} \in \mathrm{Mor}(\cc C),  \twothreemap{f} \in I \wedge \twothreemap{f} \circ f \in I \}$$
\end{dfn}

\begin{lem}\label{lemAleph}
	Let $\cc C$ be a cocomplete category. Let $I$ be a class of maps in $\cc C$. Then
	$$\twothree(I)\cofreg \subseteq \twothree(I\cofreg).$$
\end{lem}

We will first show how to prove Theorem \ref{ThmCheckOnJPrime} using Lemma \ref{lemAleph}, and then we will prove Lemma \ref{lemAleph}.

\begin{proof}[Proof of Theorem \ref{ThmCheckOnJPrime}]
	Let $TC$ denote the set of all trivial cofibrations in $\cc M$.
	We first claim that $$TC \otimes \cc M \subseteq \twothree((J^\prime \otimes \cc M)\cof).$$
	Take a trivial cofibration $f : A \rightarrow B$ and take $X \in \cc M$. We need to show $f \otimes X \in \twothree((J^\prime \otimes \cc M)\cof)$.
	Factor the map $B \rightarrow 1$ into a trivial cofibration $h : B \rightarrow B^f$ and a fibration $B^f \rightarrow 1$. So $B^f$ is fibrant. Then $h$ and $h \circ f$ are trivial cofibrations with fibrant codomain.
	By Lemma \ref{LemmaJcof} we have $h \in J^\prime\cof$ and $h \circ f \in J^\prime\cof$.
	Using the closed monoidal structure of $\cc M$ we can show that
	$$(J^\prime \cof) \otimes \cc M \subseteq (J^\prime \otimes \cc M)\cof.$$
	So it then follows that $h\otimes X$ and $(h \circ f) \otimes X$ lie in $(J^\prime \otimes \cc M)\cof$.
	Since $(h\otimes X )\circ (f \otimes X) = (h \circ f) \otimes X$ this then implies that $f \otimes X \in \twothree((J^\prime \otimes \cc M)\cof)$.
	This shows the claim that $$TC \otimes \cc M \subseteq \twothree((J^\prime \otimes \cc M)\cof).$$
	
	From Lemma \ref{lemAleph} it now follows that
		$$(TC \otimes \cc M)\cofreg \subseteq \twothree((J^\prime \otimes \cc M)\cof)\cofreg \subseteq$$  $$\subseteq \twothree((J^\prime \otimes \cc M)\cof\cofreg) \subseteq \twothree((J^\prime \otimes \cc M)\cof) $$
		
	Let us now prove the monoid axiom.
	Take $f \in (TC \otimes \cc M)\cofreg$. We need to show that $f$ is a weak equivalence. By the above subset inclusions we have $f \in \twothree((J^\prime \otimes \cc M)\cof)$. So there exists $h \in (J^\prime \otimes \cc M)\cof$, such that $h$ is composable with $f$ and $h \circ f \in (J^\prime \otimes \cc M)\cof$.
	By \cite[Corollary 2.1.15]{hovey2007model}, the maps $h$ and $h\circ f$ are then retracts from morphisms in $(J^\prime \otimes \cc M)\cofreg$.
	Since all morphisms from $(J^\prime \otimes \cc M)\cofreg$ are weak equivalences, it follows that $h$ and $h\circ f$ are weak equivalences.
	By $2$-of-$3$ it follows that $f$ is a weak equivalence.
	So $\cc M$ satisfies the monoid axiom.
\end{proof}

\begin{proof}[Proof of Lemma \ref{lemAleph}]
	Since $I \subseteq I\cofreg$ we have that $$\twothree(I) \subseteq \twothree(I\cofreg).$$
	To show the lemma we now just need to show that $\twothree(I\cofreg)$ is closed under base change and transfinite composition.
	
	\underline{Base change:}\\
	Consider a pushout diagram
	$$\xymatrix{A \ar[d] \ar[r]^f & B \ar[d] \\
	            X \ar[r]^g & Y}$$
    with $f \in \twothree(I\cofreg)$. Then there exists a morphism $\twothreemap{f} : B \rightarrow \twothreecone{f}$ such that $\twothreemap{f} \in I\cofreg$ and $\twothreemap{f} \circ f \in I\cofreg$.
    Consider the diagram
    $$\xymatrix{A \ar[d] \ar[r]^f & B \ar[d] \ar[r]^{\twothreemap{f}} & \twothreecone{f} \ar[d] \\
    	X \ar[r]^g & Y \ar[r]^(0.4)k & \twothreecone{f} \underset{B}{\coprod} Y }$$
    where the square on the right is a pushout. Then $k$ is a base change of $\twothreemap{f}$, so $k \in I\cofreg$. Also the whole diagram is a pushout, so $k \circ g$ is a base change of $\twothreemap{f} \circ f$, so $k \circ g \in I\cofreg$. So $g \in \twothree(I\cofreg)$.
    Therefore $\twothree(I\cofreg)$ is closed under base change.
    
    \underline{Transfinite composition}:\\
    Let $\alpha$ be an ordinal and let
    $$\xymatrix{X_0 \ar[r]^{f_0} & X_1 \ar[r]^{f_1} & X_2 \ar[r] & \dots \ar[r] & X_\alpha }$$
    be a sequence of morphisms in $\twothree(I\cofreg)$. So for every $\beta \leq \alpha$ we have an object $X_\beta$, such that if $\beta \leq \alpha$ is a limit ordinal then $X_\beta = \colim{\gamma < \beta} X_\gamma$, and for every $\beta < \alpha$ we have a morphism $f_\beta : X_\beta \rightarrow X_{\beta +1}$ such that $f_\beta \in \twothree(I\cofreg)$.
    
    We now define via transfinite recursion a diagram in which each square is a pushout square and where $g_{0,\gamma} = f_{\gamma}$ for every ordinal $\gamma$, 
    $$\xymatrix{ Y_{0,0} \ar[r]^{g_{0,0}} & Y_{0,1} \ar[d]^{k_{0,1}} \ar[r]^{g_{0,1}} & Y_{0,2} \ar[d]^{k_{0,2}} \ar[r]^{g_{0,2}} & \dots \ar[d] \ar[r]  & Y_{0,\alpha} \ar[d]^{k_{0,\alpha}} \\
     & Y_{1,1}  \ar[r]^{g_{1,1}} & Y_{1,2} \ar[d]^{k_{1,2}} \ar[r]^{g_{1,2}} & \dots \ar[d] \ar[r]  & Y_{1,\alpha} \ar[d]^{k_{1,\alpha}} \\
     & & Y_{2,2} \ar[r]^{g_{2,2}} & \dots \ar[d] \ar[r]  & Y_{2,\alpha} \ar[d]^{k_{2,\alpha}} \\ 
     & & & \dots \ar[r] & \dots \ar[d] \\
     & & & & Y_{\alpha,\alpha} }$$
 and where each $k_{\beta,\gamma}$ is in $I\cofreg$ and for every $\gamma$ we have $k_{\gamma, \gamma+1} \circ g_{\gamma,\gamma} \in I\cofreg$. In a moment we will spell out in detail how to construct this diagram using two nested transfinite recursions. Before we spell it out in detail we quickly outline informally how the diagram is constructed: $g_{0,0}$ is in $\twothree(I\cofreg)$, so there exists $k_{0,1}$ like in the above diagram, with $k_{0,1}$ and $k_{0,1} \circ g_{0,0}$ in $I\cofreg$. Using $k_{0,1}$ one can take all pushout squares in the top row of the above diagram and get all the $k_{0,\gamma}$ and $g_{1,\gamma}$. Then $g_{1,1}$ is a pushout of $g_{0,1}$. Since $\twothree(I\cofreg)$ is stable under base change, it follows that $g_{1,1}$ is in $\twothree(I\cofreg)$, and the process can be repeated. That is the informal description. Now comes the formal description.

    Via transfinite recursion on an ordinal $\gamma$, we define
    \begin{enumerate}
    	\item For all ordinals $\gamma \leq \beta \leq \alpha$ an object $Y_{\gamma,\beta}$.
    	\item For all ordinals $\gamma \leq \beta < \alpha$ a morphism $g_{\gamma,\beta} : Y_{\gamma,\beta} \rightarrow Y_{\gamma,\beta+1}$.
    	\item If $\gamma$ is a successor ordinal, and $\gamma \leq \beta\leq \alpha$ a morphism $k_{\gamma-1,\beta} :Y_{\gamma-1,\beta} \rightarrow Y_{\gamma,\beta}$.
    \end{enumerate} We define them such that they satisfy the following properties
    \begin{enumerate}
    	\item For all $\gamma \leq \beta < \alpha$ we have $g_{\gamma,\beta} \in \twothree(I\cofreg)$.
    	\item If $\gamma$ is a successor ordinal and $\gamma \leq \beta \leq \alpha$ then $k_{\gamma-1,\beta}\in I\cofreg$ and $k_{\gamma-1,\gamma}\circ g_{\gamma-1,\gamma-1} \in I\cofreg$.
    \end{enumerate}
    
    \underline{$\gamma$ Induction Start}: $\gamma= 0$\\
    We define $Y_{0,\beta} := X_\beta$ and $g_{0,\beta} := f_\beta$.
    Then $g_{0,\beta} \in \twothree(I\cofreg)$.
    This completes the $\gamma$ induction start.
    
    \underline{$\gamma$ Induction Step}: $\gamma = \delta+1$ for some $\delta$ satisfying the induction hypothesis.\\

    For all $\beta > \delta$ we now define by transfinite recursion over $\beta$:
    \begin{enumerate}
    	\item An object $Y_{\gamma,\beta}$
    	\item A morphism $k_{\delta,\beta}: Y_{\delta,\beta} \rightarrow Y_{\gamma,\beta}$
    	\item If $\beta$ is a successor ordinal larger than $\gamma$, a morphism $g_{\gamma,\beta-1} : Y_{\gamma,\beta-1} \rightarrow Y_{\gamma,\beta}$
    \end{enumerate}
	such that
	
	\begin{enumerate}
		\item If $\beta$ is a limit ordinal, then $Y_{\gamma,\beta}$ is the filtered colimit of the sequence 
		$$\xymatrix{ Y_{\gamma,\gamma} \ar[r]^{g_{\gamma,\gamma}} \ar[r] & Y_{\gamma,\gamma+1} \ar[r]^{g_{\gamma,\gamma+1}}  & \dots \ar[r] & Y_{\gamma,\epsilon} \ar[r]^{g_{\gamma,\epsilon}} & \dots    } .$$
		\item For all $\beta > \delta$ we have $k_{\delta,\beta} \in I\cofreg$.
		\item If $\beta$ is a successor ordinal larger than $\gamma$ then $g_{\gamma,\beta-1} \in \twothree(I\cofreg)$.
		\item $k_{\delta,\gamma} \circ g_{\delta,\delta} \in I\cofreg$.
	\end{enumerate}
	The last property does not depend on $\beta$, and will be shown in the induction start of this recursion over $\beta$.
	
	This entire transfinite induction on $\beta$ will happen within the induction step over $\gamma$. And once the this transfinite induction over $\beta$ is complete, the induction step for $\gamma$ is complete.
	
    The smallest number greater than $\delta$ is $\gamma$, so the induction over $\beta$ will start at $\gamma$.
    
    \underline{$\beta$ Induction Start}: $\beta = \gamma$\\
    
    By inductive assumption on on the $\gamma$ variable we know that $g_{\delta,\delta}$ is already defined, and that $g_{\delta,\delta} \in \twothree(I\cofreg)$.
    So by definition of $\twothree(I\cofreg)$, there exists an object $Y_{\gamma,\gamma}$ and a morphism $$k_{\delta,\gamma} : Y_{\delta,\gamma} \rightarrow Y_{\gamma,\gamma}$$ such that $k_{\delta,\gamma} \in I\cofreg$ and $k_{\delta,\gamma} \circ g_{\delta,\delta} \in I\cofreg$.
    
    This completes the $\beta$ induction start.
    
    \underline{$\beta$ Induction Step}: $\beta = \epsilon + 1$ for some $\epsilon$ satisfying the induction hypothesis.\\
    
    By inductive assumption we have already defined everything in the following diagram.
    $$\xymatrix{ Y_{\delta,\epsilon} \ar[d]_{g_{\delta,\epsilon}} \ar[r]^{k_{\delta,\epsilon}} & Y_{\gamma,\epsilon} \\
    	Y_{\delta,\beta} & }$$
    We define $Y_{\gamma,\beta}$ and $k_{\delta,\beta} : Y_{\delta,\beta} \rightarrow Y_{\gamma,\beta}$ and $g_{\gamma,\epsilon} : Y_{\gamma,\epsilon} \rightarrow Y_{\gamma,\beta}$ by taking the pushout of the above diagram.
    
    $$\xymatrix{ Y_{\delta,\epsilon} \ar[d]_{g_{\delta,\epsilon}} \ar[r]^{k_{\delta,\epsilon}} & Y_{\gamma,\epsilon}  \ar[d]^{g_{\gamma,\epsilon}} \\
    	Y_{\delta,\beta} \ar[r]_{k_{\delta,\beta}}  & Y_{\gamma,\beta} }$$
    
	Then $k_{\delta,\beta} \in I\cofreg$ because it is a base change of $k_{\delta,\delta}$.
	And $g_{\gamma,\beta} \in \twothree(I\cofreg)$, because it is a base change of $g_{\delta,\beta}$, and $\twothree(I\cofreg)$ is stable under base change.
	With this we have completed the $\beta$ induction step.
	
	\underline{$\beta$ Limit Case}: Let $\beta$ be a limit ordinal, and assume all ordinals below $\beta$ satisfy the induction hypothesis.\\
	
	Consider the following diagram of already defined objects.
	
	$$\xymatrix{ Y_{\delta,\gamma} \ar[d]_{k_{\delta,\gamma}} \ar[r]^{g_{\delta,\gamma}} \ar[r] & Y_{\delta,\gamma+1} \ar[d]_{k_{\delta,\gamma+1}} \ar[r]^{g_{\delta,\gamma+1}}  & \dots \ar[r] & Y_{\delta,\epsilon} \ar[d]_{k_{\delta,\epsilon}} \ar[r]^{g_{\delta,\epsilon}} & \dots  \\  
	Y_{\gamma,\gamma} \ar[r]^{g_{\gamma,\gamma}} \ar[r] & Y_{\gamma,\gamma+1} \ar[r]^{g_{\gamma,\gamma+1}}  & \dots \ar[r] & Y_{\gamma,\epsilon} \ar[r]^{g_{\gamma,\epsilon}} & \dots } $$
	where $\epsilon$ ranges over all ordinals less than $\beta$.
	
	By inductive assumption on $\beta$, the filtered colimit of the upper row is $Y_{\delta,\beta}$.
	We define $Y_{\gamma,\beta}$ as the filtered colimit of the lower row.
	By the universal property of colimits we get a canonical map $k_{\delta,\beta} : Y_{\delta,\beta} \rightarrow Y_{\gamma,\beta}$.
	Since all squares in the above diagram are pushout squares, the map $k_{\delta,\beta} : Y_{\delta,\beta} \rightarrow Y_{\gamma,\beta}$ is a base change of $k_{\delta,\epsilon}$ for any $\epsilon < \beta$. This implies that $k_{\delta,\beta}\in I\cofreg$.
	
	This completes the transfinite induction over $\beta$.\\
	And this then completes the induction step for $\gamma$.
	
    \underline{$\gamma$ Limit Case}: Let $\gamma$ be a limit ordinal, and assume all ordinals below $\gamma$ satisfy the induction hypothesis.\\
    For all $\beta \geq \gamma$ we define $Y_{\gamma,\beta}$ as the filtered colimit of the sequence
    $$\xymatrix{ Y_{0,\beta} \ar[r]^{k_{0,\beta}} \ar[r] & Y_{1,\beta} \ar[r]^{k_{1,\beta}}  & \dots \ar[r] & Y_{\delta,\beta} \ar[r]^{k_{\delta,\beta}} & \dots    } $$
    where $\delta$ ranges over all ordinals less than $\gamma$.
    We define 
    
    $$g_{\gamma,\beta} : Y_{\gamma,\beta} \rightarrow Y_{\gamma,\beta+1} $$
    via the pushout
    $$\xymatrix{Y_{0,\beta} \ar[d]   \ar[r]^{g_{0,\beta}} & Y_{0,\beta+1} \ar[d] \\
    	Y_{\gamma,\beta} \ar[r]^{g_{\gamma,\beta}} & Y_{\gamma,\beta+1}   } $$
    Then $g_{\gamma,\beta} \in \twothree(I\cofreg)$, because it is a base change of $g_{0,\beta}$, and $\twothree(I\cofreg)$ is stable under base change.
	
	With this the transfinite recursion on $\gamma$ is complete.
	We now have a diagram
	
	$$\xymatrix{ Y_{0,0} \ar[r]^{g_{0,0}} & Y_{0,1} \ar[d]^{k_{0,1}} \ar[r]^{g_{0,1}} & Y_{0,2} \ar[d]^{k_{0,2}} \ar[r]^{g_{0,2}} & \dots \ar[d] \ar[r]  & Y_{0,\alpha} \ar[d]^{k_{0,\alpha}} \\
		& Y_{1,1}  \ar[r]^{g_{1,1}} & Y_{1,2} \ar[d]^{k_{1,2}} \ar[r]^{g_{1,2}} & \dots \ar[d] \ar[r]  & Y_{1,\alpha} \ar[d]^{k_{1,\alpha}} \\
		& & Y_{2,2} \ar[r]^{g_{2,2}} & \dots \ar[d] \ar[r]  & Y_{2,\alpha} \ar[d]^{k_{2,\alpha}} \\ 
		& & & \dots \ar[r] & \dots \ar[d] \\
		& & & & Y_{\alpha,\alpha} }$$
	in which the top horizontal row is our initial sequence
	$$\xymatrix{X_0 \ar[r]^{f_0} & X_1 \ar[r]^{f_1} & X_2 \ar[r] & \dots \ar[r] & X_\alpha }.$$
	Furthermore we have for all $\beta < \alpha$ that $k_{\beta,\alpha} \in I\cofreg$ and $k_{\beta,\beta+1} \circ g_{\beta,\beta} \in I\cofreg.$
	
	Let $$f : X_0 \rightarrow X_\alpha$$ be the transfinite composition of all the $f_\beta : X_\beta \rightarrow X_{\beta+1}$. Let $$k : X_\alpha = Y_{0,\alpha} \rightarrow Y_{\alpha,\alpha}$$ be the transfinite composition of all the $k_{\beta,\alpha} : Y_{\beta,\alpha} \rightarrow Y_{\beta+1,\alpha}$. Let $$\Delta: X_0 = Y_{0,0} \rightarrow Y_{\alpha,\alpha}$$
	be the transfinite composition of all the $k_{\beta,\beta+1} \circ g_{\beta,\beta} : Y_{\beta,\beta} \rightarrow Y_{\beta+1,\beta+1}$.
	Then the above diagram implies that
	$$k \circ f = \Delta.$$
	Now $k$ and $\Delta$ are transfinite compositions of morphisms from $I\cofreg$. Therefore $k \in I\cofreg$ and $\Delta \in I\cofreg$. This then implies that $f \in \twothree(I\cofreg)$.
	
	Therefore $\twothree(I\cofreg)$ is closed under transfinite compositions.
	This completes the proof of Lemma \ref{lemAleph}.
\end{proof}

\section{Monoid Axiom in Bousfield localizations}\label{sectionMonoidInBousfield}

\begin{thm}\label{ThmGeneral}
	Let $\cc M$ be a monoidal model $\cc V$-category with a set of weakly generating trivial cofibrations $J^\prime$.
	Let $S$ be a class of cofibrations with cofibrant domain in $\cc M$, such that the $\cc V$-Bousfield localization $L^{\cc V}_S\cc M$ exists.
	Assume that $J^\prime \cup (S \square I_{\cc V}) $ permits the small object argument.
	If all morphisms in
	$$((J^\prime \cup (S \square I_{\cc V}) )\otimes \cc M)\cofreg $$ are weak equivalences in $\LSM$, then $\LSM$ is a monoidal model $\cc V$-category that satisfies the monoid axiom.
\end{thm}
\begin{proof}
	By Lemma \ref{lemmaBousfield} we know that $J^\prime \cup (S \square I_{\cc V})$ is a set of weakly generating trivial cofibrations for $L^{\cc V}_S\cc M$. Theorem \ref{ThmCheckOnJPrime} now implies that $L^{\cc V}_S\cc M$ satisfies the monoid axiom.
	We now just need to show that $L^{\cc V}_S\cc M$ is a monoidal model category.
	Since $\cc M$ is a monoidal model category, and the cofibrations in $\cc M$ and $L^{\cc V}_S\cc M$ coincide, we know that $L^{\cc V}_S\cc M$ satisfies the unit axiom of monoidal model categories and we know that the pushout-product of two cofibrations is a cofibration.
	
	Let $f : A \rightarrow B$ be a cofibration and $g : C \rightarrow D$ be a trivial cofibration in $L^{\cc V}_S\cc M$. We need to show that $f \square g$ is a weak equivalence in $L^{\cc V}_S\cc M$.
	
	Consider the diagram
	$$\xymatrix{ A \otimes C \ar[d] \ar[r]^{A \otimes g} & A \otimes D \ar[d] \ar@/^2.0pc/[ddr] \\
	B \otimes C \ar[r]^(0.4)h \ar@/_1.5pc/[rrd]_{B \otimes g} & A \otimes D \underset{A \otimes C}{\coprod} B \otimes C \ar[dr]^(0.6){f \square g} \\
    & & B \otimes D } $$
	
	The morphism $h$ is a base change of $A \otimes g$. Since $g$ is a trivial cofibration and $L^{\cc V}_S\cc M$ satisfies the monoid axiom, this means that $h$ is a weak equivalence in $L^{\cc V}_S\cc M$.
	Similarly $B \otimes g$ is a weak equivalence in $L^{\cc V}_S\cc M$. So by 2-of-3 it follows that $f \square g$ is a weak equivalence in  $L^{\cc V}_S\cc M$. So $L^{\cc V}_S\cc M$ is a monoidal model category. This concludes the proof of the theorem.
\end{proof}

\end{document}